\documentclass[a4paper,10pt]{amsart}

\usepackage{graphicx}

\usepackage{cite}

\usepackage{hyperref}

\hypersetup{%
pdftitle={},
pdfsubject={Mathematics},
pdfauthor={Manfred Madritsch},
pdfkeywords={}
hyperindex=true,plainpages=false}

\usepackage{a4wide}

\usepackage{amsmath}
\usepackage{amsfonts}
\usepackage{amssymb}
\usepackage{amsthm}

\newtheorem{lem}{Lemma}[section]
\newtheorem{thm}[lem]{Theorem}
\newtheorem{prop}[lem]{Proposition}

\numberwithin{equation}{section}

\newtheorem*{cor*}{Corollary}
\newtheorem*{thm*}{Theorem}

\theoremstyle{definition}

\theoremstyle{remark}
\newtheorem{rem}[lem]{Remark}

\allowdisplaybreaks[3]

\newcommand{\Z}{\mathbb{Z}}

\newcommand{\R}{\mathbb{R}}

\newcommand{\lf}{\left\lfloor}
\newcommand{\rf}{\right\rfloor}
\newcommand{\lc}{\left\lceil}
\newcommand{\rc}{\right\rceil}
\renewcommand{\lvert}{\left\vert}
\renewcommand{\rvert}{\right\vert}

\title{Construction of normal numbers via generalized prime power sequences}

\author[M. G. Madritsch]{Manfred G. Madritsch}
\address[M. G. Madritsch]{Department for Analysis and Computational Number
  Theory\\Graz University of Technology\\A-8010 Graz, Austria}
\email{madritsch@math.tugraz.at}

\author[R. F. Tichy]{Robert F. Tichy}
\address[R. F. Tichy]{Department for Analysis and Computational Number
  Theory\\Graz University of Technology\\A-8010 Graz, Austria}
\email{tichy@tugraz.at}

\subjclass[2010]{11N37 (11A63)}

\keywords{normal number, pseudo-polynomial}

\dedicatory{Dedicated to Jean-Paul Allouche on the occasion of his 60\textsuperscript{th} birthday}

\date{\today}

\begin{document}

\begin{abstract}
In the present paper the authors construct normal numbers in base $q$
by concatenating $q$-adic expansions of prime powers $\lf\alpha
p^\theta\rf$ with $\alpha>0$ and $\theta>1$.
\end{abstract}

\maketitle

\section{Introduction}

Let $q\geq 2$ be a fixed integer and $\sigma=0.a_1a_2\dots$ be the
$q$-ary expansion of a real number $\sigma$ with $0<\sigma<1$. We
write $d_1\cdots d_\ell\in\{0,1,\dots,q-1\}^\ell$ for a block of $\ell$
digits in the $q$-ary expansion. By $\mathcal{N}(\sigma;d_1\cdots
d_\ell;N)$ we denote the number of occurrences of the block $d_1\cdots
d_\ell$ in the first $N$ digits of the $q$-ary expansion of $\sigma$.
We call $\sigma$ normal to the base $q$ if for every
fixed $\ell\geq 1$
\begin{align*}
\mathcal{R}_N(\sigma)=\mathcal{R}_{N,\ell}(\sigma)= \sup_{d_1\cdots
d_\ell}\lvert\frac{1}{N}\mathcal{N}(\sigma;d_1\cdots d_\ell;N)
  -\frac{1}{q^\ell}\rvert=o(1)
\end{align*}
as $N\rightarrow\infty$, where the supremum is taken over all
blocks $d_1\cdots d_\ell\in\{0,1,\dots,q-1\}^\ell$.

A slightly different, however equivalent definition of normal numbers is due to
Borel \cite{Borel1909:les_probabilites_denombrables} who also showed that
almost all numbers are normal (with respect to the Lebesgue measure) to any
base. However, despite their omnipresence among the reals, all numbers
currently known to be normal are established by ad hoc constructions. In
particular, we do not know whether given numbers, such as $\pi$, $e$, $\log 2$
and $\sqrt 2$, are normal.

In this paper we consider the construction of normal numbers in
base $q$ as concatenation of $q$-ary integer parts of certain
functions. A first result was achieved by
Champernowne~\cite{Champernowne1933:construction_decimals_normal}, who showed that 
\begin{align*}
0.1\,2\,3\,4\,5\,6\,7\,8\,9\,10\,11\,12\,13\,14\,15\,16\,17\,18\,19\,20\dots
\end{align*}
is normal in base $10$. This construction can be easily
generalised to any integer base $q$. Copeland and Erd{\"o}s
\cite{copeland_erdoes1946:note_on_normal} proved that
\begin{align*}
0.2\,3\,5\,7\,11\,13\,17\,19\,23\,29\,31\,37\,41\,43\,47\,53\,59\,61\,67\dots
\end{align*}
is normal in base $10$.

This construction principle has been generalized in several directions. In
particular, Dumont and Thomas \cite{Dumont_Thomas1994:modifications_de_nombres}
used transducers in order to rewrite the blocks of the expansion of a given normal
number to produce another one. Such constructions using automata yield to
$q$-automatic numbers, i.e., real numbers whose $q$-adic representation
is a $q$-automatic sequence (cf. Allouche and Shallit
\cite{allouche_shallit2003:automatic_sequences}). By these means one can show
that for instance the number
\[
\sum_{n\geq0}3^{-2^n}2^{-3^{2^n}}
\]
is normal in base 2. 

In the present paper we want to use another approach to generalize
Champernowne's construction of normal numbers. In particular, let $f$ be any
function and let $[f(n)]_q$ denote the base $q$ expansion of the integer part of
$f(n)$. Then define 
\begin{equation}\label{normal}
\begin{split}
\sigma_q&=\sigma_q(f)=
  0.\lf f(1)\rf_q\lf f(2)\rf_q\lf f(3)\rf_q \lf f(4)\rf_q \lf f(5)\rf_q \lf f(6)\rf_q \dots,
\end{split}
\end{equation}
where the arguments run through all positive integers. Champernowne's example
corresponds to the choice $f(x)=x$ in \eqref{normal}. Davenport and Erd{\"o}s 
\cite{davenport_erdoes1952:note_on_normal} considered the case where $f(x)$ is
an integer valued polynomial and showed that in this case the number
$\sigma_q(f)$ is normal. This construction was subsequently extended to
polynomials over the rationals and over the reals by Schiffer
\cite{schiffer1986:discrepancy_normal_numbers} and Nakai and 
Shiokawa~\cite{Nakai_Shiokawa1992:discrepancy_estimates_class}, who were both
able to show that $\mathcal{R}_N(\sigma_q(f))=\mathcal{O}(1/\log N)$. This
estimate is best possible as it was proved by Schiffer 
\cite{schiffer1986:discrepancy_normal_numbers}. Furthermore Madritsch et al.
\cite{Madritsch_Thuswaldner_Tichy2008:normality_numbers_generated} gave a
construction for $f$ being an entire function of bounded logarithmic order.

Nakai and Shiokawa \cite{Nakai_Shiokawa1990:class_normal_numbers} constructed
a normal number by concatenating the integer part of a pseudo-polynomial
sequence, i.e., a sequence $(\lf p(n)\rf)_{n\geq1}$ where
\begin{gather}\label{mani:pseudopoly}
  p(x)=\alpha_0 x^{\theta_0}+\alpha_1x^{\theta_1}+\cdots+\alpha_dx^{\theta_d}
\end{gather}
with $\alpha_0,\theta_0,\ldots,\alpha_d,\theta_d\in\R$, $\alpha_0>0$,
$\theta_0>\theta_1>\cdots>\theta_d>0$ and at least one
$\theta_i\not\in\Z$.

This method of construction by concatenating function values is in strong
connection with properties of $q$-additive functions. We call a function $f$ strictly
$q$-additive, if $f(0)=0$ and the function operates only on the digits of the
$q$-adic representation, i.e.,
\[
  f(n)=\sum_{h=0}^\ell f(d_h)\quad\text{ for }\quad n=\sum_{h=0}^\ell d_hq^h.
\]
A very simple example of a strictly $q$-additive function is the sum of digits
function $s_q$, defined by
\[
  s_q(n)=\sum_{h=0}^\ell d_h\quad\text{ for }\quad n=\sum_{h=0}^\ell d_hq^h.
\]

Refining the methods of Nakai and Shiokawa the first author obtained the following result.
\begin{thm*}[{\cite[Theorem 1.1]{madritsch2012:summatory_function_q}}]
Let $q\geq2$ be an integer and $f$ be a strictly $q$-additive function. If $p$ is a
pseudo-polynomial as defined in (\ref{mani:pseudopoly}), then there exists
$\varepsilon>0$ such that
\begin{gather}\label{mani:mainsum}
  \sum_{n\leq N}f\left(\lf p(n)\rf\right)
  =\mu_fN\log_q(p(N))
  +NF\left(\log_q(p(N))\right)
  +\mathcal{O}\left(N^{1-\varepsilon}\right),
\end{gather}
where
\[
\mu_f=\frac1q\sum_{d=0}^{q-1}f(d)
\]
and $F$ is a $1$-periodic function depending only on $f$ and $p$.
\end{thm*}

The aim of the present paper is to extend the above results to prime power
sequences. Let $f$ be a function and set
\begin{gather}\label{mani:tau}
\tau_q=\tau_q(f)=0.\lf f(2)\rf_q \lf f(3)\rf_q \lf f(5)\rf_q \lf f(7)\rf_q \lf f(11)\rf_q \lf f(13)\rf_q \dots,
\end{gather}
where the arguments of $f$ run through the sequence of primes.

Letting $f$ be a polynomial with rational coefficients, Nakai and Shiokawa
\cite{Nakai_Shiokawa1997:normality_numbers_generated} could show that
$\tau_q(f)$ is normal. Moreover, letting $f$ be an entire function of bounded
logarithmic order, Madritsch $et al.$
\cite{Madritsch_Thuswaldner_Tichy2008:normality_numbers_generated} showed that
$\mathcal{R}_N(\tau_q(f))=\mathcal{O}(1/\log N)$.

At this point we want to mention the connection of normal numbers with uniform
distribution. In particular, a number $x\in[0,1]$ is normal to base $q$ if and
only if the sequence $\{q^nx\}_{n\geq0}$ is uniformly distributed modulo 1
(cf. Drmota and Tichy
\cite{drmota_tichy1997:sequences_discrepancies_and}). Here $\{y\}$
stands for the fractional part of $y$. Let us mention
Kaufman \cite{kaufman1979:distribution_surd_p} and
Balog \cite{balog1985:distribution_p_heta,balog1983:fractional_part_p}, who
investigated the distribution of the fractional part of $\sqrt p$ and
$p^\theta$ respectively. Harman \cite{harman1983:distribution_sqrtp_modulo}
gave estimates for the discrepancy of the sequence $\sqrt p$. In his papers
Schoissengeier~\cite{schoissengeier1979:connection_between_zeros,
  schoissengeier1978:neue_diskrepanz_fuer} connected the estimation of the
discrepancy of $\alpha p^\theta$ with zero free regions of the Riemann zeta
function. This allowed Tolev
\cite{tolev1991:simultaneous_distribution_fractional} to consider the
multidimensional variant of this problem as well as to provide an explicit
estimate for the discrepancy. This result was improved for different special
cases by Zhai \cite{zhai2001:simultaneous_distribution_fractional}. Since the
results above deal with the case of $\theta<1$ Baker and Kolesnik
\cite{baker_kolesnik1985:distribution_p_alpha} extended these considerations to
$\theta>1$ and provided an explicit upper bound for the discrepancy in this
case. This result was improved by Cao and Zhai
\cite{cao_zhai1999:distribution_p_alpha} for $\frac53<\theta<3$. A
multidimensional extension is due to Srinivasan and Tichy
\cite{srinivasan_tichy1993:uniform_distribution_prime}.

Combining the methods for proving uniform distribution mentioned above with a
recent paper by Bergelson et al.
\cite{bergelson_kolesnik_madritsch+2012:uniform_distribution_prime} we want to 
extend the construction of Nakai and Shiokawa
\cite{Nakai_Shiokawa1990:class_normal_numbers} to prime numbers. Our first main
result is the following theorem.

\begin{thm}\label{thm:normal}
Let $\theta>1$ and $\alpha>0$. Then
\[
\mathcal{R}_N(\tau_q(\alpha x^\theta))=\mathcal{O}(1/\log N).
\]
\end{thm}

\begin{rem}
This estimate is best possible as Schiffer \cite{schiffer1986:discrepancy_normal_numbers} showed.
\end{rem}

In our second main result we use the connection of this construction of normal
numbers with the arithmetic mean of $q$-additive functions as described above. Known
results in this area are due to Shiokawa \cite{shiokawa1974:sum_digits_prime},
who was able to show the following theorem.
\begin{thm*}[{\cite[Theorem]{shiokawa1974:sum_digits_prime}}]
We have
\[
\sum_{p\leq x}s_q(p)=\frac{q-1}2\frac x{\log
  q}+\mathcal{O}\left(x\left(\frac{\log\log x}{\log x}\right)^{\frac12}\right),
\]
where the sum runs over the primes and the implicit $\mathcal{O}$-constant may
depend on $q$.
\end{thm*}

Similar results concerning the moments of the sum of digits function over
primes have been established by K\'atai \cite{katai1977:sum_digits_primes}. An
extension to Beurling primes is due to Heppner
\cite{heppner1976:uber_die_summe}. 

Let $\pi(x)$ stand for the number of primes less than or equal to
$x$. Adapting these ideas to our method we obtain the following theorem.
\begin{thm}\label{thm:summatoryfun}
Let $\theta>1$ and $\alpha>0$. Then
\[
\sum_{p\leq N}s_q(\lf\alpha p^\theta\rf)=\frac{q-1}2\pi(N)\log_qN^\theta+\mathcal{O}(\pi(N)),
\]
where the sum runs over the primes and the implicit $\mathcal{O}$-constant may
depend on $q$ and $\theta$.
\end{thm}

\begin{rem}
With simple modifications Theorem \ref{thm:summatoryfun} can be extended to
completely $q$-additive functions replacing $s_q$.
\end{rem}

The proof of the two theorems is divided in three parts. In the following
section we  rewrite both statements and state the central theorem, which
combines them and which we  prove in the rest of the paper. In Section
\ref{sec:tools} we  present all the tools we  need in the proof of the
central theorem. Finally, in Section \ref{sec:proof-prop-refm} we  proof
the theorem.

\section{Preliminaries}\label{sec:preliminaries}
Throughout the paper, an interval denotes a set
\[
  I=(\alpha,\beta]=\{x:\alpha<x\leq\beta\}
  \quad\text{with}\quad\beta>\alpha\geq\frac12.
\]
We will often subdivide a interval into smaller ones. In particular we
use the observation that if $\log(\beta/\alpha)\ll\log N$, then
$(\alpha,\beta]$ is the union of, say, $s$ intervals of the type
$(\gamma,\gamma_1]$ with $s\ll\log N$ and $\gamma_1\leq2\gamma$. Given any
complex function $F$ on $I$, we have
\begin{gather}\label{bak:intervalsplit}
\lvert\sum_{x\in I}F(x)\rvert\ll(\log
N)\lvert\sum_{\gamma<x\leq\gamma_1}F(x)\rvert,  
\end{gather}
for some such $(\gamma,\gamma_1]$.

In the proof $p$ will always denote a prime. We fix the block
$d_1\cdots d_\ell$ and write $\mathcal{N}(f(p))$ for the number of occurrences of
this block in the $q$-ary expansion of $\lfloor f(p)\rfloor$. By $\ell(m)$ we
denote the length of the $q$-ary expansion of an integer $m$.

In the first step we want to get rid of the blocks that may occur between two
expansions. To this end we define an integer $N$ by
\begin{gather}\label{mani:P}
\sum_{p\leq N-1}\ell\left(\lfloor p^\theta\rfloor\right) <L\leq
\sum_{p\leq N}\ell\left(\lfloor p^\theta\rfloor\right),
\end{gather}
where $\sum$ indicates that the sum runs over all primes. Thus we get that
\begin{equation}\label{mani:NtoP}
\begin{split}
L&=\sum_{p\leq N}\ell(\lf p^\theta\rf)+\mathcal{O}(\pi(N))+\mathcal{O}(\theta \log_q(N))\\
&=\frac{\theta}{\log q}N+\mathcal{O}\left(\frac{N}{\log N}\right).
\end{split}\end{equation}
Here we have used the prime number theorem in the form
\[
  \pi(x)=\mathrm{Li}\, x+\mathcal{O}\left(\frac x{(\log x)^G}\right),
\]
where $G$ is an arbitrary positive constant and
\[
  \mathrm{Li}\,x=\int_2^x\frac{\mathrm{d}t}{\log t}.
\]
Let $\mathcal{N}(n;d_1\cdots d_\ell)$ be the number of occurrences of the block
$d_1\cdots d_\ell$ in the expansion of $n$. Since we have fixed the block
$d_1\cdots d_\ell$ we will write $\mathcal{N}(n)=\mathcal{N}(n;d_1\cdots
d_\ell)$ for short. Then \eqref{mani:NtoP} implies that
\begin{gather}\label{mani:Ntrunc}
  \lvert\mathcal{N}(\tau_q(x^\theta);d_1\cdots d_\ell;L)-\sum_{p\leq
    N}\mathcal{N}(p^\theta)\rvert\ll\frac L{\log L}.
\end{gather}

For the next step we collect all the values that have a certain length of
expansion. Let $j_0$ be a sufficiently large integer. Then for each
integer $j\geq j_0$ we get that there exists an $N_j$ such that
\[
  q^{j-2}\leq f(N_j)<q^{j-1}\leq f(N_j+1)<q^j.
\]
We note that this is possible since $f$ asymptotically grows as its leading
coefficient. This implies that
\[
  N_j\asymp q^{\frac j\beta}.
\]
Furthermore for $N\geq q^{j_0}$
we set $J$ to be the greatest length of the $q$-ary expansions of $f(p)$ over
the primes $p\leq N$, i.e.,
\begin{gather}\label{mani:JP}
J:=\max_{p\leq N}\ell(\lfloor f(p)\rfloor)=\log_q(f(N))+\mathcal{O}(1)\asymp\log
N.
\end{gather}

In the next step we want to perform the counting by adding the leading zeroes to
the expansion of $f(p)$. For $N_{j-1}<p\leq N_j$ we may write $f(p)$ in $q$-ary
expansion, i.e., 
\begin{gather*}
f(p)=b_{j-1}q^{j-1}+b_{j-2}q^{j-2}+\dots+b_{1}q+b_{0}+b_{-1}q^{-1}+\dots.
\end{gather*}
Then we denote by $\mathcal{N}^*(f(p))$ the number of occurrences of the block
$d_1,\ldots,d_\ell$ in the string $0\cdots0b_{j-1}b_{j-2}\cdots b_1b_0$, where we
filled up the expansion with zeroes such that it has length $J$. The error of
doing so can be estimated by 
\begin{equation}\label{mani:NtoNstar}\begin{split}
0&\leq\sum_{p\leq N}\mathcal{N}^*(f(p))-\sum_{p\leq N}\mathcal{N}(f(p))\\
&\leq\sum_{j=j_0+1}^{J-1}(J-j)\left(\pi(N_{j+1})-\pi(N_{j})\right)+\mathcal{O}(1)\\
&\leq\sum_{j=j_0+2}^{J}\pi(N_{j})+\mathcal{O}(1)\ll\sum_{j=j_0+2}^{J}\frac{q^{j/\beta}}j
\ll\frac N{\log N}\ll\frac L{\log L}.\\
\end{split}\end{equation}

In the following two sections we will estimate this sum of indicator functions
in order to prove the following proposition.
\begin{prop}\label{mani:centralprop}
Let $\theta>1$ and $\alpha>0$. Then
\begin{gather}\label{mani:centralprop:statement}
\sum_{p\leq
  N}\mathcal{N}^*\left(\lf \alpha p^\theta\rf\right)=q^{-k}\pi(N)\log_qN^\theta+\mathcal{O}\left(\frac{N}{\log
  N}\right)
\end{gather}
\end{prop}

\begin{proof}[Proof of Theorem \ref{thm:normal}]
We insert \eqref{mani:centralprop:statement} into \eqref{mani:Ntrunc}
and get the desired result.
\end{proof}

\begin{proof}[Proof of Theorem \ref{thm:summatoryfun}]
For this proof we have to rewrite the statement. In particular, we use that the
sum of digits function counts the number of $1$s, $2$s, etc. and
assigns weights to them, i.e.,
\[
s_q(n)=\sum_{d=0}^{q-1}d\cdot\mathcal{N}(n;d).
\]
Thus
\begin{align*}
\sum_{p\leq N}s_q(\lf p^\theta\rf)
&=\sum_{p\leq N}\sum_{d=0}^{q-1}d\cdot\mathcal{N}(p^\theta)
=\sum_{p\leq
  N}\sum_{d=0}^{q-1}d\cdot\mathcal{N}^*(p^\theta)+\mathcal{O}\left(\frac{N}{\log
    N}\right)\\
&=\frac{q-1}2\pi(N)\log_q(N^\theta)+\mathcal{O}\left(\frac{N}{\log
    N}\right)
\end{align*}
and the theorem follows.
\end{proof}


\section{Tools}\label{sec:tools}

In this section we want to present all the tools we need on the way of proof of
Proposition \ref{mani:centralprop}. We start with an estimation which essentially goes
back to Vinogradov. This will provide us with Fourier expansions for the
indicator functions used in the proof. As usual given a real number
$y$, the expression $e(y)$ will stand for $\exp\{2\pi i y\}$.

\begin{lem}[{\cite[Lemma
    12]{vinogradov2004:method_trigonometrical_sums}}]\label{vin:lem12}
Let $\alpha$, $\beta$, $\Delta$ be real numbers satisfying
\begin{gather*}
0<\Delta<\frac12,\quad\Delta\leq\beta-\alpha\leq1-\Delta.
\end{gather*}
Then there exists a periodic function $\psi(x)$ with period $1$,
satisfying
\begin{enumerate}
\item $\psi(x)=1$ in the interval $\alpha+\frac12\Delta\leq x
  \leq\beta-\frac12\Delta$,
\item $\psi(x)=0$ in the interval $\beta+\frac12\Delta\leq x
  \leq1+\alpha-\frac12\Delta$,
\item $0\leq\psi(x)\leq1$ in the remainder of the interval
  $\alpha-\frac12\Delta\leq x\leq1+\alpha-\frac12\Delta$,
\item $\psi(x)$ has a Fourier series expansion of the form
  $$
  \psi(x)=\beta-\alpha+\sum_{\substack{\nu=-\infty\\\nu\neq0}}^\infty
   A(\nu) e(\nu x),
  $$
  where
  \begin{gather}\label{mani:A}
  \lvert A(\nu)\rvert \ll \min \left( \frac 1\nu,
  \beta-\alpha,\frac{1}{\nu^2\Delta} \right).
  \end{gather}
\end{enumerate}
\end{lem}

After we have transformed the sums under consideration into exponential sums
we want to split the interval by the following lemma.
\begin{lem}\label{lem:intervalsplit}
Let $I=(a,b]$ be an interval and $F$ be a complex function defined on $I$. If
$\log(b/a)\ll L$, then $I$ is the union of $\ell$ intervals of the type
$(c,d]$ with $\ell\ll L$ and $d\leq 2c$. Furthermore we have
\[ 
  \lvert\sum_{n\in I}F(n)\rvert\ll L\lvert\sum_{n\in(c,d]}F(n)\rvert,
\]
for some such $(c,d]$.
\end{lem}

\begin{proof}
For $i=1,\ldots,\ell$ let $I_i$ be the $\ell$ splitting intervals. Then
\begin{align*}
\lvert \sum_{n\in I}F(n)\rvert
=\lvert\sum_{i=1}^\ell\sum_{n\in I_i}F(n)\rvert
\leq \ell\max_{1\leq i\leq \ell}\lvert\sum_{n\in I_i}F(n)\rvert\ll
L\lvert\sum_{n\in I_i}F(n)\rvert 
\end{align*}
\end{proof}


We will apply the following lemma in order to estimate the occurring
exponential sums provided that the coefficients are very small. This corresponds to the
case of the most significant digits in the expansion.

\begin{lem}[{\cite[Lemma 4.19]{titchmarsh1986:theory_riemann_zeta}}]
\label{tit:lem4.19}
Let $F(x)$ be a real function, $k$ times differentiable, and satisfying $\lvert
F^{(k)}(x)\rvert\geq\lambda>0$ throughout the interval $[a,b]$. Then
\[
\lvert\int_a^be(F(x))\mathrm{d}x\rvert
\leq c(k)\lambda^{-1/k}.
\]
\end{lem}

A standard tool for estimating exponential sums over the primes is Vaughan's
identity. In order to apply this identity we have to rewrite the exponential
sum into a normal one having von Mangoldt's function as weights. Therefore let
$\Lambda$ denote von Mangoldt's function, i.e.,
\[
\Lambda(n)=\begin{cases}
\log p,&\text{if $n=p^k$ for some prime $p$ and an integer $k\geq1$;}\\
0,&\text{otherwise}.
\end{cases}
\]
In the next step 
we may subdivide this weighted exponential sum into several sums of Type I and
II. In particular, let $P\geq2$ and $P_1\leq 2P$, then we define Type
I and Type II sums by the expressions
\begin{align}
&\sum_{X<x\leq X_1}a_x\sum_{\substack{Y<y\leq Y_1\\P<xy\leq
    P_1}}f(xy)\label{type:1:sum}\quad(\text{Type I})\\
&\sum_{X<x\leq X_1}a_x\sum_{\substack{Y<y\leq Y_1\\P<xy\leq P_1}}(\log y)f(xy)\notag\\
&\sum_{X<x\leq X_1}a_x\sum_{\substack{Y<y\leq Y_1\\P<xy\leq P_1}}b_yf(xy)\label{type:2:sum}\quad(\text{Type II})
\end{align}
with $X_1\leq 2X$, $Y_1\leq 2Y$, $\lvert a_x\rvert\ll P^\varepsilon$, $\lvert
b_y\rvert\ll P^\varepsilon$ for every $\varepsilon>0$ and
\[
P\ll XY\ll P,
\]
respectively. The following lemma provides the central tool for the subdivision
of the weighted exponential sum.

\begin{lem}[{\cite[Lemma 1]{baker_kolesnik1985:distribution_p_alpha}}]
\label{bakkol:vaughan}
Let $f(n)$ be a complex valued function and $P\geq2$, $P_1\leq 2P$. Furthermore
let $U$, $V$, and $Z$ be positive numbers satisfying
\begin{gather}
2\leq U<V\leq Z\leq P,\\
U^2\leq Z,\quad 128UZ^2\leq P_1,\quad 2^{18}P_1\leq V^3.
\end{gather}
Then the sum
\[
\sum_{P\leq n\leq P_1}\Lambda(n)f(n)
\]
may be decomposed into $\ll(\log P)^6$ sums, each of which is either a Type I
sum with $Y\geq Z$ or a Type II sum with $U\leq Y\leq V$.
\end{lem}



The next tool is an estimation for the exponential sum. After subdividing the
weighted exponential sum we use Vinogradov's method in order to estimate the
occurring unweighted exponential sums.

\begin{lem}[{\cite[Lemma 6]{Nakai_Shiokawa1990:class_normal_numbers}}]
\label{nakshi:lem6}
Let $k$, $P$ and $N$ be integers such that $k\geq2$, $2\leq N\leq P$. Let
$g(x)$ be real and have continuous derivatives up to the $(k+1)$th order in
$[P+1,P+N]$; let $0<\lambda<1/(2c_0(k+1))$ and
\[
  \lambda\leq\frac{g^{(k+1)}(x)}{(k+1)!}\leq c_0\lambda
  \quad(P+1\leq x\leq P+N),
\]
or the same for $-g^{(k+1)}(x)$, and let
\[
N^{-k-1+\rho}\leq\lambda\leq N^{-1}
\]
with $0<\rho\leq k$. Then
\[
  \sum_{n=P+1}^{P+N}e(g(n))\ll N^{1-\eta},
\]
where
\begin{gather}\label{mani:eta}
\eta=\frac{\rho}{16(k+1)L},\quad
L=1+\lf\frac14k(k+1)+kR\rf,\quad
R=1+\lf\frac{\log\left(\frac1\rho k(k+1)^2\right)}{-\log\left(1-\frac1k\right)}\rf.
\end{gather}
\end{lem}

\section{Proof of Proposition \ref{mani:centralprop}}\label{sec:proof-prop-refm}

We will apply the estimates of the preceding sections in order to estimate the
exponential sums occurring in the proof. We will proceed in four steps.
\begin{enumerate}
\item In the first step we use a method of Vinogradov
  \cite{vinogradov2004:method_trigonometrical_sums} in order to rewrite the
  counting function into the estimation of exponential sums. Then we will
  distinguish two cases in the following two steps.
\item First we assume that the we are interested in a block which occurs among
  the most significant digits. This corresponds to a very small coefficient in
  the exponential sum and we may use the method of van der Corput
  (cf. \cite{graham_kolesnik1991:van_der_corputs}).
\item For the blocks occurring among the least significant digits we apply
  Vaughan's identity together with ideas from a recent paper by Bergelson
  et al. \cite{bergelson_kolesnik_madritsch+2012:uniform_distribution_prime}.
\item Finally we combine the estimates of the last two steps in order to end
  the proof.
\end{enumerate}

In this proof, the letter $p$ will always denote a prime and we set
$f(x):=\alpha x^\theta$ for short. Furthermore we set
\begin{gather}\label{mani:delta}
\delta:=\min\left(\frac14,1-\theta\right).
\end{gather}

\subsection{Rewriting the sum}\label{sec:rewriting-sum}
Throughout the rest of the paper we fix a block $d_1\cdots d_\ell$. In order to
count the occurrences of this block in the $q$-ary expansion of $\lfloor f(p)
\rfloor$ ($2\le p \le P$)  we define the indicator function
\begin{align}\label{mani:I}
\mathcal{I}(t)=\begin{cases}
  1, &\text{if }\sum_{i=1}^\ell d_iq^{-i}\leq t-\lfloor t\rfloor
     <\sum_{i=1}^\ell d_iq^{-i}+q^{-\ell};\\
  0, &\text{otherwise;}
     \end{cases}
\end{align}
which is a $1$-periodic function. Indeed, we have
\[
\mathcal{I}(q^{-j}f(n)) = 1 \Longleftrightarrow d_1\cdots d_\ell =
b_{j-1}\cdots b_{j-\ell}.
\]
Thus we can write our block counting function as follows
\begin{gather}\label{mani:NthetatoNstar}
\mathcal{N}^*(f(p))=\sum_{j=l}^J\mathcal{I}\left(q^{-j}f(p)\right).
\end{gather}


Following Nakai and Shiokawa~\cite{Nakai_Shiokawa1990:class_normal_numbers} we
want to approximate $\mathcal{I}$ from above and from below by two $1$-periodic
functions having small Fourier coefficients. In particular, we set
$H=N^{\delta/3}$ and
\begin{equation}\label{mani:abd}
\begin{split}
\alpha_-=\sum_{\lambda=1}^\ell d_\lambda q^{-\lambda}+(2H)^{-1},\quad
\beta_-=\sum_{\lambda=1}^\ell d_\lambda q^{-\lambda}+q^{-\ell}-(2H)^{-1},\quad
\Delta_-=H^{-1},\\
\alpha_+=\sum_{\lambda=1}^\ell d_\lambda q^{-\lambda}-(2H)^{-1},\quad
\beta_+=\sum_{\lambda=1}^\ell d_\lambda q^{-\lambda}+q^{-\ell}+(2H)^{-1},\quad
\Delta_+=H^{-1}.
\end{split}
\end{equation}
We apply Lemma \ref{vin:lem12} with
$(\alpha,\beta,\delta)=(\alpha_-,\beta_-,\delta_-)$ and
$(\alpha,\beta,\delta)=(\alpha_+,\beta_+, \delta_+)$,
respectively, in order to get two functions $\mathcal{I}_-$ and
$\mathcal{I}_+$. By the choices of
$(\alpha_\pm,\beta_\pm,\delta_\pm)$ it is immediate that
\begin{equation}\label{uglI}
\mathcal{I}_-(t)\leq\mathcal{I}(t)\leq\mathcal{I}_+(t) \qquad
(t\in\mathbb{R}).
\end{equation}
Lemma \ref{vin:lem12} also implies that these two functions have
Fourier expansions
\begin{align}\label{mani:Ifourier}
\mathcal{I}_\pm(t)=q^{-\ell}\pm H^{-1}+
  \sum_{\substack{\nu=-\infty\\\nu\neq0}}^\infty A_\pm(\nu)e(\nu t)
\end{align}
satisfying
\begin{gather*}
\lvert A_\pm(\nu)\rvert
\ll\min(\lvert\nu\rvert^{-1},H\lvert\nu\rvert^{-2}).
\end{gather*}
In a next step we want to replace $\mathcal{I}$ by $\mathcal{I}_+$
in (\ref{mani:NthetatoNstar}). For this purpose we observe, using \eqref{uglI},
that
\begin{gather*}
\lvert\mathcal{I}(t)-\mathcal{I}_+(t)\rvert \le
\lvert\mathcal{I}_+(t)-\mathcal{I}_-(t)\rvert
  \ll H^{-1} + \sum_{\substack{\nu=-\infty\\\nu\neq0}}^\infty
  A_\pm(\nu)e(\nu t).
\end{gather*}
Thus subtracting yields the main part, and summing over $p\leq N$ gives
\begin{gather}\label{mani:0.5}
\lvert\sum_{p\leq N}\mathcal{I}(q^{-j}f(p))-\frac{\pi(N)}{q^{\ell}}\rvert
\ll\pi(N)H^{-1}+\sum_{\substack{\nu=-\infty\\\nu\neq0}}^\infty
A_{\pm}(\nu)\sum_{p\leq N}e\left(\frac{\nu}{q^j}f(p)\right).
\end{gather}

Now we consider the coefficients $A_\pm(\nu)$. Noting
\eqref{mani:A} one observes that
\begin{gather*}
A_\pm(\nu)\ll\begin{cases}
  \nu^{-1},       &\text{for }\lvert\nu\rvert\leq H;\\
  H\nu^{-2}, &\text{for }\lvert\nu\rvert>H.
          \end{cases}
\end{gather*}
Estimating all summands with $\lvert\nu\rvert>H$ trivially we get
\begin{gather*}
\sum_{\substack{\nu=-\infty\\\nu\neq0}}^\infty
  A_\pm(\nu)e\left(\frac{\nu}{q^j}f(p)\right)
\ll\sum_{\nu=1}^{H}\nu^{-1}e\left(\frac{\nu}{q^j}f(p)\right)+H^{-1}.
\end{gather*}
Using this in \eqref{mani:0.5} yields
\begin{gather}\label{mani:1.5}
\lvert\sum_{p\leq N}\mathcal{I}(q^{-j}f(p))-\frac{\pi(N)}{q^{\ell}}\rvert
\ll\pi(N)H^{-1}+\sum_{\nu=1}^{H}
\nu^{-1}\sum_{p\leq N}e\left(\frac{\nu}{q^j}f(p)\right).
\end{gather}

Finally we sum over all $j$s and get
\begin{equation}\label{mani:2}
\begin{split}
\lvert\sum_{p\leq N}\mathcal{N}^*(f(p))-\frac{\pi(N)}{q^{\ell}}J\rvert
\ll\pi(N)H^{-1}J+\sum_{j=\ell}^J\sum_{\nu=1}^{H}
\nu^{-1}S(N,j,\nu),
\end{split}
\end{equation}
where we have set
\[
S(N,j,\nu):=\sum_{p\leq N}e\left(\frac{\nu}{q^j}f(p)\right).
\]

The crucial part is the estimation of the exponential sums over the
primes. In the following we will distinguish two cases according to the size of
$j$. This corresponds to the position in the expansion of $f(p)$. In
particular, let $\rho>0$ be arbitrarily small then we want to distinguish
between the most significant digits and the least significant digits,
i.e., between the ranges
\[
1\leq q^j\leq N^{\theta-1+\rho}
\quad\text{and}\quad
N^{\theta-1+\rho}<q^j\leq N^\theta.
\]

\subsection{Most significant digits}
In this subsection we assume that
\[
N^{\theta-1+\rho}<q^j\leq N^\theta,
\]
which means that we deal with the most significant digits in the expansion. We
start by rewriting the sum into an integral.
\begin{align*}
S(N,j,\nu)=\sum_{p\leq N}e\left(\frac{\nu}{q^j}f(p)\right)
=\int_{2}^{N}e\left(\frac{\nu}{q^j}f(t)\right)\mathrm{d}\pi(t)+\mathcal{O}(1).
\end{align*}
In the second step we then apply the prime number theorem. Thus
\begin{align*}
S(N,j,\nu)
=\int_{N(\log N)^{-G}}^{N}
e\left(\frac{\nu}{q^j}f(t)\right)
\frac{\mathrm{d}t}{\log t}
+\mathcal{O}\left(\frac{N}{(\log N)^G}\right).
\end{align*}
Now we use the second mean-value theorem together with Lemma \ref{tit:lem4.19}
and $k=\lf\theta\rf$ to get
\begin{equation}\label{mani:res:most}
\begin{split}
S(N,j,\nu)&\ll\frac1{\log N}\sup_{\xi}
  \lvert\int_{N(\log N)^{-G}}^{\xi}e\left(\frac{\nu}{q^j}f(t)\right)\mathrm{d}t\rvert
  +\mathcal{O}\left(\frac{N}{(\log N)^G}\right)\\
&\ll\frac1{\log N}\left(\frac{\lvert \nu\rvert}{q^j}\right)^{-\frac1k}
  +\mathcal{O}\left(\frac{N}{(\log N)^G}\right).
\end{split}
\end{equation}

\subsection{Least significant digits}
For the digits in this range we want to apply Vaughan's identity in order to
transfer the sum over the primes into two special types of sums involving
products of integers. Before we may apply Vaughan's identity we have to weight
the exponential sum under consideration by the von Mangoldt function. By an
application of Lemma \ref{lem:intervalsplit}, it suffices to consider an
interval of the form $(P,2P]$. Thus
\[
\lvert S(N,j,\nu)\rvert
\ll(\log N)\lvert\sum_{P<p\leq2P}e\left(f(p)\right)\rvert.
\]
Using partial summation we get
\[
\lvert S(N,j,\nu)\rvert
\ll(\log N)
\lvert\sum_{P<p\leq 2P}e\left(f(p)\right)\rvert
\ll (\log N)P^{\frac12}+(\log N)\lvert\sum_{P<n\leq P_1}\Lambda(n)e\left(f(n)\right)\rvert
\]
for some $P_1$ with $P<P_1\leq 2P$. From now on we may assume that
$P>N^{1-\eta}$. 



Then an application of Lemma \ref{bakkol:vaughan} with $U=P^{\frac\delta3}$,
$V=P^{\frac13}$, $Z=P^{\frac12-\frac\delta3}$ yields 
\begin{align}\label{mani:afterVaughan}
S(N,j,\nu)\ll P^{\frac12}+\left(\log P\right)^7\lvert S_1\rvert,
\end{align}
where $S_1$ is either a Type I sum as in \eqref{type:1:sum} with
$Y\geq P^{\frac12-\frac\delta3}$ or a Type II sum as in \eqref{type:2:sum} with
\[
P^{\frac\delta3}\leq Y\leq P^{\frac13}.
\]


Suppose first that $S_1$ is a Type II sum, i.e.,
\[
S_1=\sum_{X<x\leq X_1}a_x\sum_{\substack{Y<y\leq Y_1\\P<xy\leq
    P_1}}b_ye\left(f(xy)\right).
\]
Then an application of the Cauchy-Schwarz inequality yields
\begin{align*} 
\lvert S_1\rvert^2
&\leq\sum_{X<x\leq X_1}\lvert a_x\rvert^2\sum_{X<x\leq X_1}
  \lvert\sum_{\substack{Y<y\leq Y_1\\P<xy\leq P_1}}b_ye\left(\frac{\nu}{q^j}f(xy)\right)\rvert^2\\
&\ll XP^{2\varepsilon}\sum_{Y<y\leq Y_1}\sum_{Y<z\leq Y_1}b_y\overline{b_z}
  \sum_{\substack{X<x\leq X_1\\P<xy,xz\leq P_1}}e\left(\frac{\nu}{q^j}\left(f(xy)-f(xz)\right)\right),
\end{align*}
where we have used that $\lvert a_x\rvert\ll P^\varepsilon$. Collecting all the
terms where $y=z$ and using $\lvert b_y\rvert\ll P^\varepsilon$ yields 
\begin{gather}\label{mani:3.5}
\lvert S_1\rvert^2\ll XP^{4\varepsilon}\left(XY+\sum_{Y<y<z\leq Y_1}
  \lvert\sum_{\substack{X<x\leq X_1\\P<xy,xz\leq P_1}}e\left(\frac{\nu}{q^j}\left(f(xy)-f(xz)\right)\right)\rvert\right).
\end{gather}


There must be a pair $(y,z)$ with 
$Y<y<z<Y_1$ such that
\begin{gather}\label{mani:4.5}
\lvert S_1\rvert^2\ll  P^{2+4\varepsilon}Y^{-1}+P^{4\varepsilon}XY^2
  \lvert\sum_{X_2<x\leq X_3}e(g(x))\rvert,
\end{gather}
where $X_2=\max(X,Py^{-1})$, $X_3=\min(X_1,P_1z^{-1})$ and
\[
g(x)
=\frac{\nu}{q^j}\left(f(xy)-f(xz)\right)
=\frac{\nu}{q^j}\alpha(y^\theta-z^\theta)x^\theta.
\]

We will apply Lemma \ref{nakshi:lem6} to estimate the exponential
sum. Setting
\[k:=\lc 2\theta\rc+1
\]
we get that $g^{(k+1)}(x)\sim
\nu q^{-j}\alpha\theta(\theta-1)\cdots(\theta-k)x^{\theta-(k+1)}$.
Thus
\[
\lambda\leq\frac{g^{(k+1)}(x)}{(k+1)!}\leq c_0\lambda\quad(X_2<x\leq X_3)
\]
or similarly for $-g^{(k+1)}(x)$, where
\[
\lambda=c\nu q^{-j}\alpha(y^{\theta}-z^{\theta})X^{\theta-(k+1)}
\]
and $c$ depends only on $\theta$ and $\alpha$.


Since $\theta>1$ we get
\begin{align*}
\lambda&\geq P^{\delta-\theta}Y^{\theta-1}X^{\theta-(k+1)}\geq X^{-k-\frac12}.
\end{align*}
Similarly we obtain
\[
\lambda
\leq P^{2\delta}Y^{\theta}X^{\theta-(k+1)}
\ll P^{\theta+2\delta}X^{-(k+1)}
\leq X^{-1}.
\]
Thus we get that $X^{-k-\frac12}\leq\lambda\leq X^{-1}$. Therefore an application of Lemma \ref{nakshi:lem6} yields
\[
\sum_{X_2<x\leq X_3}e(g(x))\ll X^{1-\eta},
\]
where $\eta$ depends only on $k$ an therefore on $\theta$. Inserting this in 
\eqref{mani:4.5} we get
\begin{gather}\label{mani:res:typeII}
\lvert S_1\rvert^2\ll P^{2+4\varepsilon}Y^{-1}+P^{4\varepsilon}XY^2X^{1-\eta}
\ll P^{2+4\varepsilon}\left(P^{-\delta/3}+P^{-2\eta/3}\right).
\end{gather}

The case of $S_1$ being a type I sum is similar but simpler. We have
\begin{align*}
\lvert S\rvert
\leq\sum_{X<x\leq X_1}\lvert a_x\rvert
  \lvert\sum_{\substack{Y<y\leq Y_1\\P<xy\leq P_1}}(\log y)e\left(f(xy)\right)\rvert
\ll XP^{\varepsilon}\lvert\sum_{\substack{Y<y\leq Y_1\\P<xy\leq P_1}}(\log y)e\left(f(xy)\right)\rvert
\end{align*}
for some $x$ with $X<x\leq X_1$. By a partial summation we get
\begin{gather}\label{mani:6}
  \lvert S\rvert\ll XP^\varepsilon\log P\lvert\sum_{\substack{Y_2<y\leq
      Y_3\\P<xy\leq P_1}}e\left(f(xy)\right)\rvert
\end{gather}
for some $Y\leq Y_2<Y_3\leq Y_1$. Now we set
\[
g(y)
=f(xy)
=\frac{\nu}{q^{j}}\alpha x^\theta y^\theta.
\]

Again the idea is to apply Lemma \ref{nakshi:lem6} for the estimation of the
exponential sum. We set
\[
k:=\lc 3\theta\rc +2
\]
and get for the $k+1$-st derivative 
\[
  \lambda\leq\frac{g^{(k+1)}(x)}{(k+1)!}\leq c_0\lambda\quad(X_2<x\leq X_3)
\]
or similarly for $-g^{(k+1)}(x)$, where
\[
\lambda=c\frac{\nu}{q^j}\alpha x^{\theta}Y^{\theta-(k+1)}
\]
and $c$ again depends only on $\alpha$ and $\theta$.
We may assume that $N$ and hence $P$ is sufficiently large, then we get that
\[
Y^{-(k+1)}\ll P^{-\theta}X^{\theta}Y^{\theta-(k+1)}\leq
\lambda\leq
P^{2\delta}X^{\theta}Y^{\theta-(k+1)}\ll P^{\theta+2\delta}Y^{-(k+1)}\leq Y^{-1}.
\]
Now an application of Lemma 2.5 yields
\[
\sum_{Y_2<y\leq Y_3}e(g(y))\ll Y^{1-\eta},
\]
where $\eta$ depends only on $k$ and thus on $\theta$. Inserting this in
\eqref{mani:6}  we get
\begin{gather}\label{mani:res:typeI}
  \lvert S_1\rvert \ll(\log P)XP^\varepsilon Y^{1-\eta}\ll(\log P)P^{1+\varepsilon-\eta(1/2-\delta/3)}.
\end{gather}

Combining \eqref{mani:res:typeI} and \eqref{mani:res:typeII} in
\eqref{mani:afterVaughan} yields
\begin{equation}\label{mani:res:least}
\begin{split}
\lvert S(N,j,\nu)\rvert
&\ll P^{\frac12}+\left(\log
  P\right)^7\left(P^{1+2\varepsilon}\left(P^{-\delta/6}+P^{-\eta/3}\right)+(\log
P)P^{1+\varepsilon-\eta(1/2-\delta/3)}\right)\\
&\ll P^{\frac12}+\left(\log P\right)^8P^{1-\sigma}.
\end{split}
\end{equation}

\subsection{Conclusion}
On the one hand summing \eqref{mani:res:most} over $j$ and $\nu$ yields
\begin{align*}
&\sum_{1\leq\lvert
  \nu\rvert\leq\delta^2}\lvert\nu\rvert^{-1}\sum_{N^{\theta-\delta}<q^{j}\leq
  N^{\theta}}
  S(N,j,\nu)\\
&\quad\ll\sum_{1\leq\lvert
  \nu\rvert\leq\delta^2}\lvert\nu\rvert^{-1}\sum_{N^{\theta-\delta}<q^{j}\leq
  N^{\theta}}
  \left(\frac1{\log N}\left(\frac{\lvert \nu\rvert}{q^j}\right)^{-\frac1k}
  +\mathcal{O}\left(\frac{N}{(\log N)^G}\right)\right)\\
&\quad\ll\frac1{\log N}\sum_{1\leq\lvert
  \nu\rvert\leq\delta^2}\lvert\nu\rvert^{-1-\frac1k}\sum_{N^{\theta-\delta}<q^{j}\leq
  N^{\theta}}q^{-\frac jk}
  +\mathcal{O}\left(\frac{N}{(\log N)^{G-2}}\right)\\
&\quad\ll\frac{N}{\log N}.
\end{align*}

On the other hand in \eqref{mani:res:least} we sum over $j$ and $\nu$ and get
\[
\sum_{1\leq\lvert\nu\rvert\leq\delta^2}\lvert\nu\rvert^{-1}
\sum_{q^\ell\leq q^j\leq N^\theta}S(N,j,\nu)
\ll(\log N)^2N^{\frac12}+(\log N)^{10}N^{1-\sigma'}.
\]

Combining these estimates in \eqref{mani:2} finally yields
\begin{align*}
\lvert\sum_{p\leq
  N}\mathcal{N}^*(f(p))-\frac{\pi(N)}{q^{\ell}}J\rvert\ll\frac{N}{\log N}
\end{align*}
and the proposition is proved.

\section*{Acknowledgment}
The authors thank the anonymous referee, who read very carefully the manuscript and his/her suggestions improve considerably the presentation of the results.


\def\cprime{$'$}

\end{document}